\documentclass[11pt]{article}
\usepackage[utf8]{inputenc}
\usepackage[utf8]{inputenc}
\usepackage{hyperref}
\usepackage{pst-node}
\usepackage{enumitem}
\usepackage{auto-pst-pdf}
\usepackage{tikz-cd} 
\usepackage{comment}
\usepackage{amsmath,amssymb,amsthm}
\usepackage{amsfonts,tikz}
\usepackage[margin=1in]{geometry}
\usepackage{mathtools}
\usetikzlibrary{positioning,automata,shapes}
\tikzset{every loop/.style={min distance=10 mm, in=60, out=120, looseness=10}}
\newcommand{\bk}{\backslash}
\newcommand{\lm}{\lambda}

\newcommand{\Lm}{\Lambda}
\newcommand{\tm}{\widetilde{m}}

\newcommand{\Lift}{\textbf{\textsf{Lift}}}

\renewcommand{\epsilon}{\varepsilon}

\newtheorem{lemma}[]{Lemma}
\newtheorem{cor}[lemma]{Corollary}
\newtheorem{theorem}[lemma]{Theorem}

\newtheorem{definition}[lemma]{Definition}
\newtheorem{obs}[lemma]{Observation}

\title{Homogeneous Sets in Graphs and a \\ Chromatic Multisymmetric Function\footnote{Department of Combinatorics \& Optimization, University of Waterloo, Waterloo, ON, N2L 3G1.\newline  Emails:  lcrew@uwaterloo.ca, ehaithcock@uwaterloo.ca, jreynes@uwaterloo.ca, sspirkl@uwaterloo.ca. \newline
We acknowledge the support of the Natural Sciences and Engineering Research Council of Canada (NSERC), [funding
reference numbers RGPIN-2020-03912 and RGPIN-2022-03093]. \newline Cette recherche a \'et\'e financ\'ee par le Conseil de recherches en sciences naturelles et en g\'enie du Canada (CRSNG),
[num\'ero de r\'ef\'erence RGPIN-2020-03912 and RGPIN-2022-03093]. \\ This project was funded in part by the Government of Ontario.}}
\author{Logan Crew, Evan Haithcock, Josephine Reynes, Sophie Spirkl}
\date{\today}

\begin{document}

\maketitle

\begin{abstract}

In this paper, we extend the chromatic symmetric function $X$ to a \emph{chromatic $k$-multisymmetric function} $X_k$, defined for graphs equipped with a partition of their vertex set into $k$ parts. We demonstrate that this new function retains the basic properties and basis expansions of $X$, and we give a method for systematically deriving new linear relationships for $X$ from previous ones by passing them through $X_k$. 

In particular, we show how to take advantage of homogeneous sets of $G$ (those $S \subseteq V(G)$ such that each vertex of $V(G) \bk S$ is either adjacent to all of $S$ or is nonadjacent to all of $S$) to relate the chromatic symmetric function of $G$ to those of simpler graphs. Furthermore, we show how extending this idea to homogeneous pairs $S_1 \sqcup S_2 \subseteq V(G)$ generalizes the process used by Guay-Paquet to reduce the Stanley-Stembridge conjecture to unit interval graphs.
    
\end{abstract}

\section{Introduction}

The chromatic symmetric function $X_G$ of a graph $G$, introduced by Stanley approximately thirty years ago \cite{stanley}, has seen a recent resurgence of interest, with research focusing on generalizations, basis expansions, and its ability to distinguish graphs \cite{trees, huh, epos, dahl, foley, heil, pau, paw, wang}. In particular, a central driving conjecture in the field is the Stanley-Stembridge conjecture, which in its original form suggested that chromatic symmetric functions of incomparability graphs of $(3+1)$-free posets are $e$-positive. Substantial progress on this conjecture was made by Guay-Paquet in 2013 \cite{guay} by demonstrating a relation that expressed $X_G$ with $G$ the incomparability graph of a $(3+1)$-free poset as a convex combination of chromatic symmetric functions of incomparability graphs of posets that are simultaneously $(3+1)$-free and $(2+2)$-free, or equivalently unit interval graphs. Thus, the Stanley-Stembridge conjecture was reduced to showing the $e$-positivity of a smaller, well-studied graph class, and much recent work in the area has focused on this version of the conjecture \cite{abreu, per2022, cho2022, colm, dahlberg2019, matherne}.

In recent work by the first and last authors \cite{modular}, we extended work of Penagui\~{a}o \cite{raul} considering $X$ as a mapping from the Hopf algebra $\Gamma$ of vertex-labelled graphs to the space of symmetric functions $\Lambda$, and in doing so we gave a characterization of all local graph modifications (written as a linear combination of vertex-labelled induced graphs) that universally preserve the chromatic symmetric function. Notably, it is possible to show that Guay-Paquet's relation in \cite{guay} is not one of these, meaning that it depends on the particular structure of the incomparability graphs of $(3+1)$-free posets.

In this work, we define a further extension of the chromatic symmetric function to multiple sets of variables, also known as a \emph{multisymmetric function}. The different sets of variables will represent a partition of $V(G)$ into nonempty parts, where each part gets its own variable set. This allows us to generalize results of \cite{modular} and \cite{raul} and characterize further graph modifications that preserve the chromatic symmetric function. In particular, while not every linear combination $L$ of chromatic symmetric functions lying in the kernel of the map $X$ represents a universal graph modification, we show that every such linear combination \emph{does} naturally give rise to a family of graphs within which $L$ always represents such a graph modification.

In particular, the chromatic multisymmetric function captures the importance of \emph{homogeneous partitions} of a graph $G$, meaning partitions $V(G) = V_1 \sqcup \dots \sqcup V_k \sqcup W$ such that for every vertex $w \in W$ and every $i$, either $w$ is adjacent to every vertex of $V_i$, or no vertex of $V_i$. The notion of homogeneous partitions occurs naturally in structural graph theory; the particular case of homogeneous pairs (where $k=2$) occurs in the original form of the decomposition theorem that underlies the famous proof of the Strong Perfect Graph Theorem by Chudnovsky, Robertson, Seymour, and Thomas \cite{perfect}, while homogeneous pairs of cliques (where $V_1$ and $V_2$ are both cliques) play a vital role in the structure theorem of Chudnovsky and Seymour \cite{chud} for claw-free graphs, which in particular include all incomparability graphs of $(3+1)$-free posets.

As an example, note that if a poset is not $(2+2)$-free, then its incomparability graph contains an induced four-vertex cycle, or $C_4$. We show how Guay-Paquet's relation reducing the Stanley-Stembridge conjecture \cite{guay} may be viewed naturally in terms of chromatic multisymmetric functions, as it takes advantage of the nontrivial fact (implicitly proved by Guay-Paquet's structure theorem with Morales and Rowland in \cite{clawposet}, and directly proved in Section 5.3) that if $G$ is the incomparability graph of a poset that is $(3+1)$-free but not $(2+2)$-free, then for each induced $C_4$ in $G$, there exists a homogeneous pair of cliques such that each clique contains two vertices of the $C_4$. Guay-Paquet used this to his advantage in \cite{guay} by finding an appropriate local relation on the subgraph induced by these cliques to show that the chromatic symmetric function of the original graph is equal  to a convex combination of chromatic symmetric functions of graphs in which the $C_4$ is eliminated.

This paper is organized as follows: in Section 2, we introduce the notation, terminology, and basic ideas needed from symmetric function theory and graph theory. In Section 3, we introduce \emph{$k$-vertex-labelled graphs}, defined by labelling the vertices of a graph $G$ with one of $k$ labels, and thus inducing a partition of its vertex set $V(G)$ into $k$ parts. We extend the definition of the chromatic symmetric function $X$ to a \emph{chromatic $k$-multisymmetric function} $X_k$ on such partitioned graphs, and demonstrate that this function has properties and basis expansions naturally generalizing those of $X$. 

In Section 4, we build on \cite{modular} and \cite{raul} by characterizing the kernel of the functions $X_k$. Then in Section 5 we determine the algebraic relationships between different $X_k$, including that elements in $Ker(X_k)$ may be projected to elements of $Ker(X)$, or in some cases lifted to elements of $Ker(X_{k+1})$. We put this all together to show how to derive further elements of $Ker(X)$ from a given one (and thus better describe how graphs can have equal chromatic symmetric function) in a systematic way, and show examples from the literature that can be recovered in this manner. Finally, in Section 6 we provide some further possible directions for research.

\section{Background}

\subsection{Fundamentals of Partitions and Symmetric Functions}

 A \emph{set partition} of a set $S$ is a collection of nonempty, pairwise nonintersecting \emph{blocks} $B_1, \dots, B_k$ satisfying $B_1 \cup \dots \cup B_k = S$. We will specify that a union of blocks is a set partition by writing $\sqcup$ for disjoint union, using the notation $B_1 \sqcup \dots \sqcup B_k$.
 
 An \emph{integer partition} is a tuple $\lm = (\lm_1,\dots,\lm_k)$ of positive integers such that $\lm_1 \geq \dots \geq \lm_k$.  The integers $\lm_i$ are the \emph{parts} of $\lm$.  If $\sum_{i=1}^k \lm_i = n$, we say that $\lm$ is a partition of $n$. The number of parts equal to $i$ in $\lm$ is given by $n_i(\lm)$.
 
 We may use simply \emph{partition} to refer to either a set or integer partition. We write $\pi \vdash S$ or $\lm \vdash n$ to mean respectively that $\pi$ is a partition of $S$, and $\lm$ is a partition of $n$, and we write $|\lm| = |\pi| = n$. The number of blocks or parts is the \emph{length} of a partition, and is denoted by $l(\pi)$ or $l(\lm)$. When $\pi$ is a set partition, we will write $\lm(\pi)$ to mean the integer partition whose parts are the sizes of the blocks of $\pi$.

  A function $f(x_1,x_2,\dots) \in \mathbb{C}[[x_1,x_2,\dots]]$ is \emph{symmetric}\footnote{The choice of coefficient ring is irrelevant for the work in this paper so long as it is a field of characteristic $0$.} if $f(x_1,x_2,\dots) = f(x_{\sigma(1)},x_{\sigma(2)},\dots)$ for every permutation $\sigma$ of the positive integers $\mathbb{N}$.  The \emph{algebra of symmetric functions} $\Lambda$ is the subalgebra of $\mathbb{C}[[x_1,x_2,\dots]]$ consisting of those symmetric functions $f$ that are of bounded degree (that is, there exists a positive integer $n$ such that every monomial of $f$ has degree $\leq n$).  Furthermore, $\Lambda$ is a graded algebra, with natural grading
  $$
  \Lambda = \bigoplus_{d=0}^{\infty} \Lambda^d
  $$
  where $\Lambda^d$ consists of symmetric functions that are homogeneous of degree $d$. For more on the basics of symmetric function theory see \cite{mac,stanleybook}.

  Each $\Lambda^d$ is a finite-dimensional vector space over $\mathbb{C}$, with dimension equal to the number of integer partitions of $d$ (and thus, $\Lambda$ is an infinite-dimensional vector space over $\mathbb{C}$).  Some commonly-used bases of $\Lambda$ that are indexed by partitions $\lm = (\lm_1,\dots,\lm_k)$ include:
\begin{itemize}
  \item The monomial symmetric functions $m_{\lm}$, defined as the sum of all distinct monomials of the form $x_{i_1}^{\lm_1} \dots x_{i_k}^{\lm_k}$ with distinct indices $i_1, \dots, i_k$.

  \item The power-sum symmetric functions, defined by the equations
  $$
  p_n = \sum_{k=1}^{\infty} x_k^n, \hspace{0.3cm} p_{\lm} = p_{\lm_1}p_{\lm_2} \dots p_{\lm_k}.
  $$
  \item The elementary symmetric functions, defined by the equations
  $$
  e_n = \sum_{i_1 < \dots < i_n} x_{i_1} \dots x_{i_n}, \hspace{0.3cm} e_{\lm} = e_{\lm_1}e_{\lm_2} \dots e_{\lm_k}.
  $$
\end{itemize}

  We also make use of the \emph{augmented monomial symmetric functions}, defined by 
  $$
  \tm_{\lm} = \left(\prod_{i=1}^{\infty} n_i(\lm)!\right)m_{\lm}.
  $$
  
  Given a symmetric function $f$ and a basis $b$ of $\Lambda$, we say that $f$ is \emph{$b$-positive} if when we write $f$ in the basis $b$, all coefficients are nonnegative.
  
  \subsection{Fundamentals of Graphs and Colorings}
  
  We use standard graph terminology as in \cite{modular}.

  A \emph{graph} $G = (V,E)$ consists of a \emph{vertex set} $V$ and an \emph{edge multiset} $E$ where the elements of $E$ are (unordered) pairs of (not necessarily distinct) elements of $V$. Given an edge $e \in E$, its two vertices are called its \emph{endpoints}. An edge $e \in E$ that contains the same vertex twice is called a \emph{loop}. If there are two or more edges that each contain the same two vertices, they are called \emph{multi-edges}. A graph is called \emph{simple} if its edge multiset contains no loops or multi-edges. 
  
  Given a graph $G = (V,E)$ and $S \subseteq V$, let $E|_S$ denote the set of edges of $G$ with both endpoints in $S$. The graph $G|_S = (S, E|_S)$ is called the subgraph of $G$ \emph{induced} by $S$. A graph $H$ is said to be an \emph{induced subgraph} of $G$ if there exists a set $S \subseteq V$ such that $H$ is isomorphic to $G|_S$, and in this case we say that $G|_S$ is an \emph{induced} (copy of) $H$ in $G$. If $H$ is not an induced subgraph of $G$, we say that $G$ is \emph{$H$-free}.
  
  A \emph{complete graph} is a simple graph such that for each pair of distinct vertices $u,v \in V(G)$, $uv \in E(G)$. Given a simple graph $G$, its \emph{complement} $\overline{G}$ is the graph $V(G), \overline{E(G)}$ where for each pair of distinct vertices $u,v \in V(G)$, we have $uv \in \overline{E(G)} \iff uv \notin E(G)$. Given graphs $G$ and $H$, the \emph{disjoint union} $G \sqcup H$ is equal to $(V(G) \sqcup V(H), E(G) \sqcup E(H))$.
  
  Given $A, B \subseteq V(G)$ with $A \cap B = \emptyset$, we say that $B$ is \emph{complete to $A$} if for every $a \in A$ and $b \in B$, $ab \in E(G)$. We say that $B$ is \emph{anticomplete to $A$} if for every $a \in A$ and $b \in B$, $ab \notin E(G)$.

  Given a graph $G$, there are two commonly used operations that produce new graphs. One is \emph{deletion}: given an edge $e \in E(G)$, the graph of $G$ \emph{with} $e$ \emph{deleted} is the graph $G' = (V(G), E(G) \bk \{e\})$, and is denoted $G \bk e$ or $G-e$. Likewise, if $S$ is a multiset of edges, we use $G \bk S$ or $G-S$ to denote the graph $(V(G),E(G) \bk S)$.

  The other operation is the \emph{contraction of an edge} $e = v_1v_2$, denoted $G / e$.  If $v_1 = v_2$ ($e$ is a loop), we define $G / e = G \bk e$.  Otherwise, we create a new vertex $v^*$, and define $G / e$ as the graph $G'$ with $V(G') = (V(G) \bk \{v_1,v_2\}) \cup v^*$, and $E(G') = (E(G) \bk E(v_1, v_2)) \cup E(v^*)$, where $E(v_1,v_2)$ is the set of edges with at least one of $v_1$ or $v_2$ as an endpoint, and $E(v^*)$ consists of each edge in $E(v_1,v_2) \bk \{e\}$ with the endpoint $v_1$ and/or $v_2$ replaced with the new vertex $v^*$.  Note that this is an operation on a graph that identifies two vertices while keeping and/or creating multi-edges and loops.

  Let $G = (V(G),E(G))$ be a graph. A map $\kappa: V(G) \rightarrow \mathbb{N}_{>0}$ is called a \emph{coloring} of $G$. This coloring is called \emph{proper} if $\kappa(v_1) \neq \kappa(v_2)$ for all $v_1,v_2$ such that there exists an edge $e = v_1v_2$ in $E(G)$. The \emph{chromatic symmetric function} $X_G$ of $G$ is defined as \cite{stanley}
  
  $$X_G(x_1,x_2,\dots) = \sum_{\kappa \text{ proper}} \prod_{v \in V(G)} x_{\kappa(v)} = \sum_{\pi \text{ stable}} \tm_{\lm(\pi)}$$ 
  where the first sum ranges over all proper colorings $\kappa$ of $G$, the second sum ranges over all (set) partitions $\pi$ of $V(G)$ into stable sets, and $\lm(\pi)$ is the integer partition whose parts are \\ $\{|\pi_i|: \pi_i \text{ is a block of } \pi\}$. Note that if $G$ contains a loop then $X_G = 0$, and that $X_G$ is unchanged by replacing each multi-edge by a single edge. 
  
\subsection{Vertex-Weighted Graphs and their Colorings}

A \emph{vertex-weighted graph} $(G,w)$ consists of a graph $G$ and a weight function $w: V(G) \rightarrow \mathbb{N}_{>0}$. For $S \subseteq V(G)$, denote $w(S) = \sum_{v \in S} w(v)$. 


Given a vertex-weighted graph $(G,w)$, if $e = v_1v_2$ is a non-loop edge, we define the contraction of $G$ by e to be the graph $(G/e,w/e)$, where $w/e$ is the weight function such that $(w/e)(v) = w(v)$ if $v$ is not the vertex $v^*$ arising from the contraction, and $(w/e)(v^*) = w(v_1) + w(v_2)$ (if $e$ is a loop, we define $w/e = w$, so $(G/e,w/e) = (G \bk e, w)$).

The chromatic symmetric function may be extended to vertex-weighted graphs as
$$
X_{(G,w)} = \sum_{\kappa \text{ proper}} \prod_{v \in V(G)} x_{\kappa(v)}^{w(v)} = \sum_{\pi \text{ stable}} \tm_{\lm(\pi)}
$$
where again the sum ranges over all proper colorings $\kappa$ of $G$, and $\lm(\pi)$ is the integer partition whose parts are $\{w(\pi_i): \pi_i \text{ is a block of } \pi\}$. In this setting the chromatic symmetric function admits the deletion-contraction relation \cite{delcon}
\begin{equation}\label{eq:delcon}
X_{(G,w)} = X_{(G \bk e, w)} - X_{(G/e,w/e)}.
\end{equation}

\section{Extending $X_{(G,w)}$ to a Multisymmetric Function}

Previous work \cite{delcon} has dealt with extending $X_G$ to vertex-weighted graphs using positive integer weights in order to express a deletion-contraction relation for the chromatic symmetric function. Here, we make a further extension to allow for graphs whose weights are \emph{tuples} of nonnegative integers to allow us to systematically describe a family of chromatic symmetric function relations including that of Guay-Paquet \cite{guay}. To do so, we need to introduce and describe the vector space of \emph{multisymmetric functions}. We describe only the results we need here; for more information see the foundational works of Dalbec \cite{dalbec} and Vaccarino \cite{vaccarino2005ring}.

\subsection{Multisymmetric Functions}

\begin{definition}[\cite{dalbec, vaccarino2005ring}]
    Let $k$ be a fixed positive integer, and for $i = 1, \dots, k$, let $X^i = \{(x_1)_i, (x_2)_i, \dots\}$ be a set of countably many commuting indeterminates. A function $f \in \mathbb{C}[X^1, \dots, X^k]$ is \emph{multisymmetric} if for all $\sigma: \mathbb{Z}^+ \rightarrow \mathbb{Z}^+$, $f$ is unchanged by replacing each $(x_i)_j$ by $(x_{\sigma(i)})_j$ (that is, $f$ is fixed under the diagonal action of $S_{\mathbb{Z}^+}$ on the $k$ variable sets simultaneously). 
    
    \medskip
    
    We denote the vector space of multisymmetric functions in $k$ sets of variables (or $k$-multisymmetric functions) by $\Lm_k$.
\end{definition}

As a vector space, there is a natural grading 
$$
\Lm_k = \bigoplus_{i=0}^{\infty} \Lm_k^i
$$
where $\Lm_k^i$ consists of those $k$-multisymmetric functions that are homogeneous of total degree $i$. This may be further decomposed as
$$
\Lm_k^i = \bigoplus_{(i_1, \dots, i_k)} \Lm_k^{(i_1, \dots, i_k)}
$$
where the direct sum ranges over all elements $(i_1, \dots, i_k)$ of $\mathbb{Z}_{\geq 0}^k$ such that $\sum_j i_j = i$ (in other words, all weak compositions of $i$ with $k$ parts), and $\Lm_k^{(i_1, \dots, i_k)}$ is the vector space of $k$-multisymmetric functions in which every monomial has total degree $i_j$ in the variable set $X^j$.

Many symmetric function bases have analogues in multisymmetric functions. Where we index basis elements of $\Lambda^i$ with a multiset of positive integers summing to $i$, we index basis elements of $\Lambda_k^i$ with a multiset of \emph{ordered $k$-tuples of nonnegative integers} (where each tuple has at least one positive coordinate) such that the sum of all coordinates of all tuples sums to $i$. We call such multisets \emph{$k$-tuple partitions}. Where we list integer partitions with their parts in decreasing order, we will use the notation $\lm^k$ to denote a generic $k$-tuple partition $\lm^k = (\lm_1^k, \dots, \lm_l^k)$ with $\lm_1^k \geq \dots \geq \lm_l^k$, where each $\lm_i^k$ is an element of $\mathbb{Z}_{\geq 0}^k \bk \{0^k\}$ and $\geq$ is the reverse lexicographic order. For instance, a basis element of $\Lambda^5$ might be indexed by $(4,1)$ or $(2,2,1)$, whereas a basis element of $\Lambda_2^5$ might be indexed by $((2,2),(1,0))$, or $((1,1),(1,1),(0,1))$, and a basis element of $\Lambda_3^5$ might be indexed by $((2,2,0),(0,1,0))$ or $((2,2,1))$.

Furthermore, $\Lambda_k^{(i_1, \dots, i_k)}$ has basis elements indexed by $k$-tuple partitions such that the componentwise sum of all of the tuples is $(i_1, \dots, i_k)$, so for example $((2,2,0),(0,1,0))$ would index a basis element of $\Lambda_3^{(2,3,0)}$.  In analogy with usual symmetric functions, we define $|\lm^k| = \sum_i \lm_i^k$ (note that this is a $k$-tuple), $||\lm^k||$ is the sum of all integers in all $\lm_i^k$, $l(\lm^k)$ is the number of tuples of $\lm^k$, and for any $k$-tuple $\alpha$, $n_{\alpha}(\lm^k)$ is the multiplicity of $\alpha$ as a tuple of $\lm^k$.

For example $((1,2), (1,1), (0,2))$ is a $2$-tuple partition, and we have $|((1,2), (1,1), (0,2))| = (2,5)$ and $||((1,2), (1,1), (0,2))|| = 7$. We also have $l((1,2), (1,1), (0,2)) = 3$, and \\ $n_{(1,2)}((1,2),(1,1),(0,2)) = 1$, while $n_{(1,0)}((1,2),(1,0),(1,0),(0,1)) = 2$.

Throughout this paper, we will use the shorthand $x_i^{(j_1, \dots, j_k)} \cong ((x_i)_1)^{j_1} \dots ((x_i)_k)^{j_k}$. We will often use $\alpha \in \mathbb{Z}_{\geq 0}^k \bk \{0^k\}$ to denote a $k$-tuple, and we let $\epsilon_i^k$ denote the particular $k$-tuple with $i^{th}$ coordinate equal to $1$, and all others equal to $0$ (the superscript $k$ may be dropped when it is clear from context).

The following functions indexed by $k$-tuple partitions each give bases of $\Lm_k^{(i_1, \dots, i_k)}$ when taken over all $k$-tuple partitions $\lm^k$ such that $|\lm^k| = (i_1, \dots, i_k)$ \cite{dalbec}:

\begin{itemize}
  \item The \emph{monomial $k$-multisymmetric functions} $m_{\lm^k}$, defined as the sum of all distinct monomials of the form $x_{i_1}^{\lm_1^k} \dots x_{i_l}^{\lm_l^k}$ with distinct indices $i_1, \dots, i_l$.  For example, if $k = 2$ we have
  $$
  m_{((1,2),(0,1))} = \sum_{i \neq j} (x_i^{(1,2)}x_j^{(0,1)}) = \sum_{i \neq j} (x_i)_1(x_i)_2^2(x_j)_2.
  $$
  
  \item The \emph{augmented} monomial $k$-multisymmetric functions $\tm_{\lm^k}$, defined by
  $$
  \tm_{\lm^k} = \left(\prod_{\alpha} n_{\alpha}(\lm^k)!\right) m_{\lm^k}.
  $$

  \item The \emph{power-sum $k$-multisymmetric functions}, defined by the equations
  $$
  p_{((i_1, \dots, i_k))} = \sum_{j=1}^{\infty} x_j^{(i_1, \dots, i_k)}, \hspace{0.3cm} p_{\lm^k} = p_{\lm_1^k}p_{\lm_2^k} \dots p_{\lm_l^k}.
  $$
  
  For example,
  $$
  p_{((1,2),(0,1))} = \left(\sum_i (x_i)_1(x_i)_2^2\right)\left(\sum_j (x_j)_2\right).
  $$
  \item  The \emph{elementary $k$-multisymmetric functions}, defined by the equations
  $$
  e_{((i_1, \dots, i_k))} = \tm_{{\epsilon_1}^{i_1} \dots {\epsilon_k}^{i_k}}, \hspace{0.3cm} e_{\lm^k} = e_{\lm_1^k}e_{\lm_2^k} \dots e_{\lm_l^k}.
  $$
  Where $\epsilon_j^{i_j}$ means $i_j$ copies of $\epsilon_j$. For example,
  $$
  e_{((1,2),(0,1))} = (\tm_{((1,0),(0,1),(0,1))})(\tm_{((0,1))}) = \left(2 \sum_{i_1,i_2,i_3 \text{ distinct}} (x_{i_1})_1(x_{i_2})_2(x_{i_3})_2\right)\left(\sum_j (x_j)_2\right).
  $$
\end{itemize}

\subsection{Chromatic Multisymmetric Functions of Weighted Graphs}

We will now extend the chromatic symmetric function of integer-weighted graphs given in \cite{delcon} to graphs where the vertex weights may be (non-zero) $k$-tuples of nonnegative integers:

\begin{definition}
    A \emph{tuple-weighted graph} $(G,w,k)$ consists of a graph $G$, and a weight function $w: V(G) \rightarrow \mathbb{Z}_{\geq 0}^k \setminus \{0^k\}$.
\end{definition}

\begin{definition}\label{def:gwk}
    The \emph{tuple-weighted chromatic symmetric function} of $(G,w,k)$ is defined as
    $$
     X_{(G,w,k)} = \sum_{\kappa \textnormal{ proper}} \prod_{v \in V(G)} x_{\kappa(v)}^{w(v)}.
    $$
\end{definition}

Note that despite the using near-identical notation for simplicity, the input $k$ means that this is a $k$-multisymmetric function. In order for this definition to be consistent with previous work, we use the convention that if $k$ is not given it is assumed to be $1$, in which case it is easy to verify this is just the previously-described integer-weighted chromatic symmetric function.

Before going further, the reader will naturally wonder what the motivation is for adding more variable sets to the function. The answer is that to capture the full power of certain local relationships of $X$, it is desirable for the function to have some way of detecting certain distinguished subsets of $V(G)$ in a graph $G$.

For example, in his work reducing the Stanley-Stembridge conjecture, Guay-Paquet \cite{guay} uses a chromatic symmetric function relation that holds only in graphs which have a \emph{homogeneous pair} $V_1$ and $V_2$ of cliques, meaning that $V_1$ and $V_2$ are each separately either complete or anticomplete to $V(G) \bk (V_1 \cup V_2)$. In such graphs, given a coloring of $V_1 \cup V_2$, permuting the vertices of either $V_1$ or $V_2$ does not affect how the remainder of the graph may be colored, so the effect of applying certain local graph modifications on $X_{G|_{V_1 \cup V_2}}$ can be extended naturally to examine the effect of applying the same modifications on $X_G$. This is implicitly used by Guay-Paquet in the proof reducing the Stanley-Stembridge conjecture to unit interval graphs, but is not able to be directly captured by the chromatic symmetric function. We will discuss this in greater detail in Section 5.3.

To address this idea, the motivation is that each of $V_1$ and $V_2$ should have its own variable set (and the remainder of the graph its own variable set), so that we may identify the portions of chromatic symmetric function monomials arising from $V_1$ and $V_2$. Thus, if $X^1$ and $X^2$ are the variable sets corresponding to $V_1$ and $V_2$ respectively, with $X^3$ corresponding to the remainder of the graph, in the case of ``unweighted" graphs (those in which each vertex $v$ has weight satisfying $||w(v)|| = 1$), a vertex of weight $(1,0,0)$ lies in $V_1$, a vertex of weight $(0,1,0)$ lies in $V_2$, and one of weight $(0,0,1)$ lies in $V(G) \bk (V_1 \cup V_2)$. Additionally, in these graphs a variable $x_i^{(i_1,i_2,i_3)}$ occurring in a chromatic symmetric function monomial tells us how many vertices receive the color $i$ in each of those three parts.

With this motivation in mind, we shall now show that many properties of the integer-weighted $X_{(G,w)}$ extend to the tuple-weighted $X_{(G,w,k)}$. First, note that as in the case of $X_{(G,w)}$ (see \cite{delcon}), many classical multisymmetric function bases may be written as chromatic multisymmetric functions of certain graphs:
\begin{itemize}
    \item If $K^{\lm^k}$ is the complete graph with vertices of weights $\lm_1^k, \dots, \lm_l^k$, then $X_{K^{\lm^k}} = \tm_{\lm^k}$.
    \item If $\overline{K^{\lm^k}}$ is the complement of the above graph then \begin{equation}\label{eq:pbasis}
        X_{\overline{K^{\lm^k}}} = p_{\lm^k}.
    \end{equation}
    \item If $\alpha$ is a $k$-tuple of nonnegative integers not all equal to zero, and $G^{\alpha} = K^{\epsilon_1^{\alpha_1}, \dots, \epsilon_k^{\alpha_k}}$ (where as before $\epsilon_j^{\alpha_j}$ means $\alpha_j$ copies of $\epsilon_j$), then $X_{G^{\alpha}} = e_{\alpha}$, and thus if $G^{\lm^k} = \bigsqcup_i G^{\lm_i^k}$ then $X_{G^{\lm^k}} = e_{\lm^k}$.
\end{itemize}

For each of the following theorems, we also note the corresponding result for the integer-weighted chromatic symmetric function and (where it exists) for the usual chromatic symmetric function.

\begin{theorem}
    If $(G,w,k)$ is a vertex-weighted graph, and $e$ is an edge of $G$, then
    \begin{equation}\label{eq:delcon}    X_{(G,w,k)} = X_{(G-e,w,k)} - X_{(G/e,w/e,k)}
    \end{equation}
    where $w/e = w$ if $e$ is a loop of $G$, and otherwise if $e = v_1v_2$ and the newly-formed vertex is $v^*$, we have $(w/e)(v) = w(v)$ for $v \neq v^*$, and $(w/e)(v^*) = w(v_1)+w(v_2)$ where addition is componentwise.
    \label{theorem:delcon}
\end{theorem}

\begin{proof}
    This proof is very similar to that of \cite[Lemma 2]{delcon}. We will show that for every choice of $e\in E(G)$ there is a bijection between the set of proper colorings of $G-e$ and the set of proper of colorings of either $G$ or $G/e$. To show this we demonstrate that each proper coloring of $G-e$ corresponds to a proper coloring of either $G/e$ or $G$, but not both. 
    
    Let $e=v_1v_2$. If $v_1=v_2$, then $e$ is a loop and the statement is trivial. 
    Suppose that $v_1\neq v_2$ and $\kappa$ is a proper coloring of $(G-e,w,k)$. Let $v^*$ be the label of the vertex obtained by contracting $e$. We consider two cases, either $\kappa(v_1)=\kappa(v_2)$ or $\kappa(v_1)\neq\kappa(v_2)$.

    If $\kappa(v_1)=\kappa(v_2)$, then this coloring does not correspond to proper coloring of $(G,w,k)$. However it corresponds to a proper coloring $\kappa'$ of $(G/e,w/e,k)$ with $\kappa'(v^*)=\kappa(v_1)=\kappa(v_2)$ and $\kappa'(u)=\kappa(u)$ for all $u\in  V(G/e)\bk \{v^*\}$ . By construction $(w/e)(v^*)=w(v_1)+w(v_2)$. Notice that $x_{\kappa(v_1)}^{w(v_1)}x_{\kappa(v_2)}^{w(v_2)}=x_{\kappa(v_1)}^{w(v_1)}x_{\kappa(v_1)}^{w(v_2)}=x_{\kappa'(v^*)}^{w(v_1)+w(v_2)}=x_{\kappa'(v^*)}^{(w/e)(v^*)}$ and thus the corresponding monomials from $\kappa$ and $\kappa'$ are equal in the chromatic multisymmetric function. 
    
    If $\kappa(v_1)\neq\kappa(v_2)$, then $\kappa$ is a proper coloring of $(G,w,k)$ and $(G-e,w,k)$, contributing the same term to each chromatic multisymmetric function as they have the same weight functions, but does not correspond to a proper coloring of $(G/e,w/e,k)$. 
    
    In either case there exists a corresponding proper coloring of either $G/e$ or $G$, but not both, and this coloring yields same term in the chromatic multisymmetric function. Note that this gives a one-to-one correspondence between proper colorings of $G-e$ and the set of proper colorings of $G$ and $G/e$.

    Thus $X_{(G,w,k)}+X_{(G/e,w/e,k)}=X_{(G-e,w,k)}$ and the result holds.

\end{proof}

\begin{lemma}\label{lem:mbasis}

$$
X_{(G,w,k)} = \sum_{\substack{\pi \vdash V(G) \\ \pi \textnormal{ stable}}} \tm_{\lm^k(\pi)}
$$
where $\lm^k(\pi)$ is the $k$-tuple partition whose parts are the total weights of the blocks of $\pi$.

\end{lemma}

\begin{proof}

This proof is analogous to that of \cite[Lemma 1]{delcon}. First observe that given a proper coloring $\kappa$ of $(G,w,k)$ using $l$ distinct colors $j_1 < \dots < j_l$, $V(G)$ can be partitioned into $L_1,\dots, L_l$ such that for every $v\in V(G)$, $v\in L_i$ if and only if $\kappa(v)=j_i$. This is by definition a $k$-tuple partition $\pi=L_1\sqcup\dots\sqcup L_l$ of $V(G)$ into stable sets. So each proper coloring corresponds to a stable set partition. 

Using Definition \ref{def:gwk}, $\kappa$ corresponds to the monomial $x_{\kappa}=\prod_{v \in V(G)} x_{\kappa(v)}^{w(v)}$. Permuting the assignment of colors to the $L_i$, where colors are only permuted amongst $L_i$ of the same weight, produces all colorings corresponding to $x_{\kappa}$. Thus the number of proper colorings of $G$ that correspond to $\kappa$ is the number of ways to permute the parts of $\pi$ with the same weight. 

Summing over all colorings that give a distinct color to each part of $\pi$ yields a symmetric function of monomials of type $\lambda^k(\pi)$. This is an $m$-basis element of type $\lm^k(\pi)$ with the coefficient given by the number of colorings $\kappa$ that yield the set partition $\pi$. There are $\prod_{i=1}^\infty n_i(\lambda^k(\pi))!$ ways to permute the $L_i$ by weight and thus this many colorings for each stable set partition $\pi$. Since by definition $\tm_{\lm^k} = (\prod_{i=1}^\infty n_i(\lambda^k(\pi))!)m_{\lm^k}$, we have\
$$X_{(G,w,k)}=\sum_{\kappa \text{ proper}} \prod_{v \in V(G)} x_{\kappa(v)}^{w(v)}=\sum_{\substack{\pi \vdash V(G) \\ \pi \text{ stable}}} \tm_{\lm^k(\pi)}.$$

\end{proof}

\begin{lemma}\label{lem:pbasis}
$$
X_{(G,w,k)} = \sum_{S \subseteq E(G)} (-1)^{|S|}p_{\lm^k(S)}
$$
where $\lm^k(S)$ is the $k$-tuple partition whose parts are the total weights of the connected components of $(V(G),S)$ with vertex weighting $w$.

\end{lemma}

\begin{proof}
    This proof is an adaptation of the proof of \cite[Lemma 3]{delcon}. We begin by ordering the edges of $(G,w,k)$ as $f_1,f_2,...,f_m$. Now we apply Theorem \ref{theorem:delcon} repeatedly as follows. First we observe that $X_{(G,w,k)}=X_{(G-{f_1},w,k)}-X_{(G/f_1,w/f_1,k)}$ and we can apply the deletion-contraction to both $(G-f_1,w,k)$ and $(G/f_1,w/f_1,k)$ where we delete and contract $f_2$ to get a new equation for $X_{(G,w,k)}$ with four terms from of deletion and contraction of $f_1$ and then $f_2$. We repeat this process on each new term where edge $f_i$ is deleted and contracted from each term at the $i$th step from all $2^{i-1}$ terms is the equation for $X_{(G,w,k)}$. This process terminates after $m$ iterations. The final function will be of the form 
    \begin{equation*}
        X_{(G,w,k)}=\sum_{S\subseteq E(G)}(-1)^{|S|}X_{(G(S),w(S),k)}
    \end{equation*}
    where $(G(S),w(S),k)$ is the graph where the edges in $S$ are contracted and the edges in $E(G)\setminus S$ are deleted. Note that this graph has no edges and is thus the complement of a complete graph. For each vertex $v$ in $G(S)$ where $y_1, \dots y_m$ are the vertices of $G$ contracted to $v$, we have $w(v)=\sum_{y_i} w(y_i)$, so as given by equation \eqref{eq:pbasis}, we see that $X_{(G(S),w(S),k)}=p_{\lambda^k(S)}$ and thus the result holds. 

\end{proof}

\section{The Kernel of the Chromatic Multisymmetric Function}

In \cite{raul}, Penagui\~{a}o considered $X$ as a map from the Hopf algebra $\Gamma$ of vertex-labelled graphs to the space of symmetric functions. If $\Gamma$ is extended to include vertex-weighted graphs, then it is easy to verify that the kernel of the map from this space to symmetric functions is generated by vertex relabellings and the deletion-contraction relation. Penagui\~{a}o resolves the more difficult case of restricting $\Gamma$ to unweighted graphs by showing that in this case the kernel of $X$ of is generated by vertex relabellings and the triangular modular relation introduced by Orellana and Scott in \cite{ore}. 

In this section, we show that the kernel of the chromatic $k$-multisymmetric function on unweighted graphs is generated by a very similar set of relations. In the multisymmetric setting, unweighted graphs still have $k$ distinct types of vertices that can occur, one for each weight $\epsilon_i$ for $i = 1, \dots, k$; this can be viewed as representing a partition of $V(G)$ into $k$ nonempty blocks $V_1 \sqcup \dots \sqcup V_k$, where the vertices of weight $\epsilon_i$ are in block $i$. Thus, we work with formal linear combinations of graphs where the vertices are split into $k$ groups, and then each group separately is labelled. The following definition fixes how we will label a graph with $\alpha_i$ vertices of weight $\epsilon_i$ for each $i$.

\begin{definition}
    A \emph{$k$-vertex-labelled graph $G$} is one in which the vertices $V(G)$ are labelled by
    \begin{itemize}
        \item Choosing a $k$-tuple $(\alpha_1, \dots, \alpha_k)$ of nonnegative integers such that $\sum_i \alpha_i = |V(G)|$;
        \item Labelling the vertices $(i,j)$ where $i$ ranges from $1$ to $k$ (skipping those with $\alpha_i = 0$) and $j$ (depending on $i$) ranges from $1$ to $\alpha_i$; and
        \item Assigning the weight $\epsilon_i$ to vertices with first coordinate $i$. The set of vertices of weight $\epsilon_i$ will be denoted $V_i$, so $V = \bigsqcup_{i=1}^k V_i$.
    \end{itemize} 
\end{definition}

\begin{definition}
    The algebra $\Gamma_k$ consists of formal linear combinations of $k$-vertex-labelled graphs.
\end{definition}

Now, for the sake of clarity we let $X_k : \Gamma_k \rightarrow \Lm_k$ represent the map defined by letting $X_k(G)$ be the chromatic $k$-multisymmetric function of $G$ for each $k$-vertex-labelled graph, and extending linearly. 

In \cite{raul}, Penagui\~{a}o proved the aforementioned representation of $Ker(X)$ by showing that any graph may be written using these relations as a linear combination of \emph{complete multipartite graphs}, graphs which admit a partition of the vertex set into maximal stable sets $S_1, \dots, S_k$, meaning that each $S_i$ is complete to each $S_j$ with $i \neq j$ (equivalently, a complete multipartite graph is the complement of a disjoint union of cliques). We emulate this proof and show that an analogous result holds in the case of $k$-multisymmetric functions.

\begin{definition}
    For $\alpha \in \mathbb{Z}_{\geq 0}^k-\{0^k\}$, define $I^{\alpha}$ to be the graph with $||\alpha||$ vertices such that exactly $\alpha_i$ vertices have weight $\epsilon_i$ and no edges (thus $I^{\alpha}$ is an empty graph that is $k$-vertex-labelled $V_1 \sqcup \dots \sqcup V_k$ with $|V_i| = \alpha_i$). For an integer $k$-partition $\lm^k$, define $I^{\lm^k}$ to be the graph with vertex set $V(I^{\lm_1^k}) \sqcup \dots \sqcup V(I^{\lm_{l(\lm)}^k})$ and that contains an edge $v_kv_l$ for each pair of vertices $v_k, v_l$ that are not in of the same $I^{\lm_i^k}$, and no other edges (in graph theoretic terms, $I^{\lm^k}$ is the complete multipartite graph formed as the complete join of the $I^{\lm_i^k}$).
    We define $r_{\lm^k}$ to be the chromatic $k$-multisymmetric function of $I^{\lm^k}$.
\end{definition}

\begin{lemma}

The set $\{r_{\lm^k} : |\lm^k| = (i_1, \dots, i_k)\}$ is a basis for $\Lm_k^{(i_1, \dots, i_k)}$.

\end{lemma}

\begin{proof}

Note that the stable sets of $I^{\lm^k}$ are subsets of the vertex sets $V(I^{\lm^k_j})$, $1\leq j \leq l(\lm^k)$. Let $\pi^* \vdash V(I^{\lm^k})$ be the stable partition of $I^{\lm^k}$ with the smallest number of blocks, so $\pi^* = \sqcup_j \, V(I^{\lm_j^k})$. Moreover, for set partitions $\pi_1, \pi_2 \vdash V(I^{\lm^k})$, we say $\pi_1 \leq \pi_2$ if $\pi_1$ is a refinement of $\pi_2$ as a set partition. 
 From Lemma \ref{lem:mbasis}, we see that
$$
r_{\lm^k} = \sum_{\substack{\pi \vdash V(I^{\lm^k}) \\ \pi \text{ stable}}} \tm_{\lm^k(\pi)} 
= \sum_{\pi \leq \pi^*} \tm_{\lm^k(\pi)}.
$$
So each $r_{\lm^k}$ can be expressed as a linear combination of the $\tm$ functions. Furthermore, consider the basis transition matrix from $r_{\lm^k}$ to $\tm_{\lm^k}$, with the rows and columns each indexed by all integer $k$-partitions listed in reverse lexicographic order. It is straightforward to see that this matrix is upper triangular and non-zero on the diagonal, so is invertible. 
Hence, $\{r_{\lm^k} : |\lm^k| = (i_1, \dots, i_k)\}$ is a basis for $\Lm_k^{(i_1, \dots, i_k)}$.
\end{proof}

\begin{cor}
The set $\{r_{\lm^k} : ||\lm^k|| = m\}$ is a basis for $\Lm_k^m$, and the set $\{r_{\lm^k}\}$ is a basis for $\Lm_k$.
\end{cor}

We now will reiterate some definitions from \cite{modular} we will need:

\begin{definition}
Let $H_1, \dots, H_m$ be $k$-vertex-labelled graphs with the same vertex set. Given a linear combination $L = c_1H_1 + \dots + c_mH_m \in \Gamma_k$ and a $k$-vertex-labelled graph $G$ such that $V(G) \supseteq V(H_i)$, we define the \emph{extension of $L$ by $G$} as 
$$
Ext(L;G) = c_1(G \uplus E(H_1)) + \dots + c_m(G \uplus E(H_m)).
$$
where $\uplus$ represents that we add all edges of $E(H_i)$ to $G$, possibly resulting in multi-edges.
\end{definition}

We introduce extensions because it will be convenient to define certain key elements of $Ker(X_k)$ by finding a linear combination $L \in Ker(X_k)$ of small graphs such that $Ext(L;G) \in Ker(X_k)$ for all appropriate $G$; intuitively $L$ represents a local modification that always preserves the chromatic $k$-multisymmetric function inside of a larger graph.

\begin{definition}\label{def:rela}
\phantom{ }

\begin{itemize}
 \item For a $k$-vertex-labelled graph $G$ with vertex set $V = \sqcup_{i=1}^k V_i$, let $S_G = S_{V_1} \times \dots \times S_{V_k}$. Given $\sigma \in S_G$, let $G_{\sigma}$ denote the $k$-vertex-labelled graph arising from $G$ by applying $\sigma$ to $V$. Then for all $G_{\sigma}$,
$$
\ell_{iso}(G,\sigma) = G - G_{\sigma} \in \Gamma_k.
$$
\item For an ordered triple $t = (m,n,p)$ of positive integers, let $a_t = (m,i_m)$, $b_t = (n,i_n)$, $c_t=(p,i_p)$, where $i_m = 1$, $i_n = 1+\delta_{mn}$, and $i_p = 1 + \delta_{mp}+\delta_{np}$. Let $H_t$ be the graph with $V(H_t) = \{a_t,b_t,c_t\}$, and $E(H_t) = \{a_tb_t,a_tc_t,b_tc_t\}$. Define
$$\ell_{os(t)} = H_t - H_t \bk \{a_tb_t\} - H_t \bk \{a_tc_t\} + H_t \bk \{a_tb_t,a_tc_t\}.$$
\end{itemize}

\begin{figure}[hbt]
\begin{center}
  \begin{tikzpicture}[scale=1.5]
    \node[label=below:{$b_t$}, fill=black, circle] at (0, 0)(1){};
    \node[label=below:{$c_t$}, fill=black, circle] at (1, 0)(2){};
    \node[label=above:{$a_t$}, fill=black, circle] at (0.5, 0.866)(3){};

    \draw[black, thick] (3) -- (1);
    \draw[black, thick] (3) -- (2);
    \draw[black, thick] (2) -- (1);

    \draw (1.5, 0.5) coordinate (MI) node[right] { $\bf{-}$ };

    \node[label=below:{$b_t$}, fill=black, circle] at (2, 0)(4){};
    \node[label=below:{$c_t$}, fill=black, circle] at (3, 0)(5){};
    \node[label=above:{$a_t$}, fill=black, circle] at (2.5, 0.866)(6){};
    
    \draw[black, thick] (5) -- (4);
    \draw[black, thick] (6) -- (4);
    
    \draw (3.5, 0.5) coordinate (MI2) node[right] { $\bf{-}$ };

    \node[label=below:{$b_t$}, fill=black, circle] at (4, 0)(7){};
    \node[label=below:{$c_t$}, fill=black, circle] at (5, 0)(8){};
    \node[label=above:{$a_t$}, fill=black, circle] at (4.5, 0.866)(9){};

    \draw[black, thick] (7) -- (8);
    \draw[black, thick] (8) -- (9);
    
    \draw (5.5, 0.5) coordinate (PL) node[right] { $\bf{+}$ };

    \node[label=below:{$b_t$}, fill=black, circle] at (6, 0)(10){};
    \node[label=below:{$c_t$}, fill=black, circle] at (7, 0)(11){};
    \node[label=above:{$a_t$}, fill=black, circle] at (6.5, 0.866)(12){};

    \draw[black, thick] (10) -- (11);

  \end{tikzpicture}
\end{center}
\label{fig:rel2}
\caption{$\ell_{os(t)}$}
\end{figure}
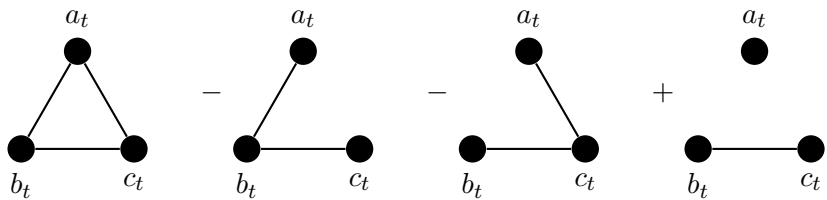

Furthermore, define $T_{iso}(k) = \{\ell_{iso}(G,\sigma): G \textnormal{ a } k \textnormal{-vertex-labelled graph}, \sigma \in S_G\}$, and $T_{os}(k) = \\ \{Ext(\ell_{os(t)};G) : G \textnormal{ a } k \textnormal{-vertex-labelled graph}, t \in \{1,2,\dots,k\}^3\}$.

\end{definition}

Intuitively, the first part of this definition is telling us that given a $k$-vertex-labelled graph, we can rearrange its vertices in any way that preserves the first coordinates of the labels to get the same chromatic $k$-multisymmetric function. The second part of the definition tells us that we may extend the Orellana-Scott modular relation given in \cite{ore} to $k$-multisymmetric functions (hence the notation $\ell_{os(t)}$).

We will now characterize $Ker(X_k)$. First, we prove a well-known auxiliary lemma. A $K_1 \sqcup K_2$ in a graph $G$ is a subgraph $G|_{\{v_1,v_2,v_3\}}$, where $v_1, v_2, v_3 \in V(G)$ are such that exactly one of $v_1v_2$, $v_1v_3$, $v_2v_3$ is an element of $E(G)$. If $G$ contains no such three vertices, it is said to be \emph{$K_1 \sqcup K_2$-free}.

\begin{lemma}[Folklore]\label{lem:aux}

A graph is $K_1 \sqcup K_2$-free if and only if it is a complete multipartite graph.

\end{lemma}

\begin{proof}
We first prove that all complete multipartite graphs are $K_1 \sqcup K_2$-free. Suppose otherwise for a contradiction; let $G$ be a complete multipartite graph containing $K_1 \sqcup K_2$ as an induced subgraph. Say $v_1,v_2,v_3 \in V(G)$ make up this subgraph with the edge $v_1v_2 \in E(G)$. Since $v_1, v_2$ are adjacent, they lie in different maximal stable sets. But since $v_3$ is non-adjacent to both $v_1$ and $v_2$, it lies in the same maximal stable set as both $v_1$ and $v_2$, a contradiction.

For the other direction, let $G$ be a minimal counterexample. Note that if $G$ contains no edges, then it is complete multipartite with a single stable set. Otherwise, let $U \subset V(G)$ be a maximal stable set and consider a vertex $v$ outside of $U$. $v$ must be adjacent to a vertex in $U$, otherwise $U$ is not maximal. But then $v$ must be adjacent to every vertex in $U$, since otherwise there is an induced $K_1 \sqcup K_2$. So $V(G) \setminus U$ is complete to $U$. Lastly, note that the graph induced by $V(G) \setminus U$ is complete multipartite by the minimality of $G$. Hence, $G$ is complete multipartite.
\end{proof}

\begin{theorem}\label{thm:kerk}
$$
span(T_{iso}(k),T_{os}(k)) = Ker(X_k).
$$
\end{theorem}

\begin{proof}
We first show $span(T_{iso}(k),T_{os}(k)) \subset Ker(X_k)$. Note that $T_{iso}(k) \subset Ker(X_k)$ as a relabelling of the graph maintaining vertex weights does not alter the chromatic multisymmetric function. Furthermore, using the same arguments as in \cite{raul}, it is easy to see that $T_{os}(k)$ is a modular relation and thus an element of the kernel.

It remains to show that $Ker(X_k) \subset span(T_{iso}(k),T_{os}(k))$. Let $a \in Ker(X_k)$ be an arbitrary element. Define $S_o$ to be the set of all linear combinations of elements of $T_{os}(k)$ and $S_i$ to be the set of all linear combinations of elements of $T_{iso}(k)$. Then we endeavor to show that there exists $o \in S_o$, $i \in S_i$ such that $a - o - i= 0$, from which it would follow that $a = o+i \in span(T_{iso}(k),T_{os}(k))$.

Recall from Lemma \ref{lem:aux} that a graph is complete multipartite if and only if it is $K_1 \sqcup K_2$-free. The key idea is to apply the modular relation to a kernel element until it contains no $K_1 \sqcup K_2$ as an induced subgraph. Then we get a linear combination of the $r$-basis elements, which we show evaluates to 0. 

Consider a linear combination of graphs equal to $a$. Let $G$ be a graph in this linear combination with a non-zero coefficient and a $K_1 \sqcup K_2$ induced subgraph $H$, and subject to these conditions, with as many non-edges as possible. Then, using the notation outlined in Definition \ref{def:rela}, there exists some $t$ such that for an element $i$ of $T_{iso}(k)$, $G-i = Ext(H_t \bk \{a_tb_t,a_tc_t\}; G')$, where $G'$ is a graph that isomorphic to $G$ but with the labels permuted. Then, $G-i-Ext(\ell_{os(t)}; G') = Ext(H_t \bk \{ a_tb_t\}; G') + Ext(H_t \bk \{ a_tc_t\}; G') - Ext(H_t; G')$, which is a linear combination of graphs with fewer non-edges. We repeat this process for each instance of an induced $K_1 \sqcup K_2$ subgraph in the linear combination. Note that the process terminates as the number of non-edges in each step is strictly decreasing, so once a graph no longer occurs with non-zero coefficient, no future step will change that.

Once this process has terminated, we apply the elements of $T_{iso}(k)$ such that each graph in the resulting linear combination is one of the $I^{\lm^k}$. Then in total, we have shown that for each $a$, there exists $o \in S_o$, $i \in S_i$ such that $a - o -i = \sum c_{\lm^k}I^{\lm^k}$, and $X_k(a - o -i) = \sum c_{\lm^k} r_{\lm^k}$. But then, since $a - o -i \in Ker(X_k)$,
$$
0 = X_k(a-o-i) = \sum c_{\lm^k} r_{\lm^k}.
$$
But the $\{ r_{\lm^k}\}$ form a basis of $\Lambda_k$, so $c_{\lm^k} = 0$ for all $k$-tuple partitions. Hence, $a - o -i = 0$ and $a \in span(T_{iso}(k),T_{os}(k))$.
\end{proof}

\section{Applications to Chromatic Symmetric Function Relations}

In this section, we apply the theory built so far to collect some previously known relations for $X_G$ under the same umbrella, and provide some extensions. To do so, we need to formally relate the $Ker(X_k)$ to $Ker(X)$ via projection.

\subsection{The Algebra in the Background}

\begin{definition}

Let $\pi_{\Gamma_k}: \Gamma_k \rightarrow \Gamma_{k-1}$ be given by extending linearly the operation that relabels a \\ $k$-vertex-labelled graph by relabelling its $\alpha_k$ vertices of weight $\epsilon_k$ as \\ $(k-1,\alpha_{k-1}+1), \dots, (k-1, \alpha_{k-1}+\alpha_k)$ (thus folding the vertex set $V_k$ into $V_{k-1}$).

Let $\pi_{\Lambda_k}: \Lambda_k \rightarrow \Lm_{k-1}$ be given by extending linearly the mapping taking $(x_i)_k$ to $(x_i)_{k-1}$.

\end{definition}

\begin{lemma}\label{lem:commute}

The following diagram commutes:
\begin{center}
\begin{tikzcd}
\Gamma_k \arrow{r}{X_k} \arrow[dr, phantom, "\circlearrowleft"] \arrow[swap]{d}{\pi_{\Gamma_{k}}} & \Lambda_k \arrow{d}{\pi_{\Lambda_k}} \\%
\Gamma_{k-1} \arrow{r}{X_{k-1}}& \Lambda_{k-1}
\end{tikzcd}
\end{center}
\end{lemma}

\begin{proof}

This follows by direct computation.

\end{proof}

\begin{cor}\label{cor:proj}

Let $L = c_1H_1 + \dots + c_mH_m \in \Gamma_k$ be a linear combination of $k$-vertex-labelled graphs such that $L \in Ker(X_k)$, and let $\pi = \pi_{\Gamma_2} \circ \dots \circ \pi_{\Gamma_k}$. Then $\pi(L) \in Ker(X)$.

\end{cor}

Thus, describing the $Ker(X_k)$ is one approach to finding more insight for $Ker(X)$. To make full use of this, we want to take relations on small graphs and use them to understand patterns on larger graphs, which we will do by adding structure around a labelled graph:

\begin{definition}

Let $H \in \Gamma_k$ be a $k$-vertex-labelled graph with vertex set $V = \bigsqcup_{i=1}^k V_i$. If $H^*$ is a $(k+1)$-vertex-labelled graph such that $V(H^*) = V \sqcup V_{k+1}$, where all vertices of $V_{k+1}$ have weight $\epsilon_{k+1}$ and all vertices of $V$ have the same weights as in $H$, we call $H^*$ an \emph{augmentation of} $H$. In this case, we define the \emph{lift} of $H$ to $H^*$ as 
$$
\Lift(H;H^*) = H^* \uplus E(H)
$$

Furthermore, we may view $\Lift$ as a map from $\Gamma_k$ to $\Gamma_{k+1}$ by extending linearly, so if $L = c_1H_1 + \dots + c_mH_m \in \Gamma_k$ is a linear combination of $k$-vertex-labelled graphs with the same vertex set with shared augmentation $H^*$, then
$$
\Lift(L;H^*) = c_1(H^* \uplus E(H_1)) + \dots + c_m(H^* \uplus E(H_m)).
$$

\end{definition}






Note that $\Lift(L;H^*)$ is defined essentially identically as $Ext(L;G)$, except that $Ext$ is an operation from $\Gamma_k$ to itself, whereas $\Lift$ is an operation from $\Gamma_k$ to $\Gamma_{k+1}$.

Corollary \ref{cor:proj} shows that $Ker(X_k)$ gets ``strictly smaller" as $k$ increases; but we can use the $\Lift$ operation to nonetheless describe certain elements of $Ker(X_k)$ for higher $k$ relative to elements of lower ones.

\begin{obs}\label{obs:oslift}

Let $o \in T_{os}(k)$ (so $o$ is a four-term linear combination of $k$-vertex-labelled graphs satisfying the triangular modular relation). Then for every augmentation $H^*$ of the shared vertex set of the graphs in $o$, we have $\Lift(o;H^*) \in T_{os}(k+1)$.

\end{obs}

\begin{obs}\label{obs:isolift}

Let $G$ be a $k$-vertex-labelled graph, and let $G-G_{\sigma} \in T_{iso}(k)$. Let $H^*$ be an augmentation of $G$, and let $\sigma'$ be the permutation of $V(H^*)$ that is the identity on elements of $V_{k+1}(H^*)$, and restricts to $\sigma$ otherwise.

Then $\Lift(G;H^*)-\Lift(G_{\sigma};H_{\sigma'}^{*}) \in T_{iso}(k+1)$.

\end{obs}

Note that although Observation \ref{obs:oslift} essentially characterizes all elements of any $T_{os}(k)$ (as would be expected since this is a modular relation for all graphs), the elements formed by Observation \ref{obs:isolift} do not capture all possibilities for $T_{iso}(k+1)$ since we have not included isomorphisms induced by permutations that act nontrivially on $V_{k+1}$.

\subsection{Chromatic Symmetric Function Relations}

We now show how to use the above observations in conjunction with Corollary \ref{cor:proj} to give a systematic way to construct more complex elements of $Ker(X)$ from simpler ones.

\begin{definition}
If $L = c_1H_1 + \dots + c_mH_m \in \Gamma_k$ is a linear combination of $k$-vertex-labelled graphs with the same vertex set $V = \bigsqcup_{i=1}^k V_i$ such that $L \in Ker(X_k)$, we say that a \emph{kernel-form presentation} of $L$ is an expression $L = I+O$, where $I \in span(T_{iso}(k))$ and $O \in span(T_{os}(k))$. We say that $S \subseteq S_{V_1} \times \dots \times S_{V_k}$ is a \emph{sufficient (permutation) set} for $L$ if there exists a kernel form presentation $L = I+O$ such that $I \in span(\{l_{iso}(G,\sigma) : G \textnormal{ a } k \textnormal{-vertex-labelled graph with }\\ V(G) = \bigsqcup_{i=1}^k V_i, \sigma \in S\})$.
\end{definition}

Note that there may be multiple distinct choices of $S$ that form a sufficient permutation set for $L$.

\begin{theorem}\label{thm:pass}

Let $L = c_1H_1 + \dots + c_mH_m \in \Gamma_k$ be a linear combination of $k$-vertex-labelled graphs with the same vertex set $V = \bigsqcup_{i=1}^k V_i$. Suppose that $L \in Ker(X_k)$, and let $S \subseteq S_{V_1} \times \dots \times S_{V_k}$ be a sufficient permutation set for $L$. 

Suppose that $H^*$ is a $(k+1)$-vertex-labelled graph that is an augmentation of the $H_i$, and suppose furthermore that for every $\sigma \in S$ we have $H^* = H_{\sigma'}^*$ as labelled graphs, where $\sigma'$ is as in Observation \ref{obs:isolift}. Then
$$
\Lift(L;H^*) \in Ker(X_{k+1}).
$$

\end{theorem}

\begin{proof}

Let us fix a kernel-form presentation $L = I + O$, where $I$ is a sum of element of $T_{iso}(k)$ and $O$ is a sum of elements of $T_{os}(k)$. By Theorem \ref{thm:kerk}, it suffices to show that for each summand $i$ of $I$, and each summand $o$ of $O$, that $\Lift(i;H^*) \in T_{iso}(k+1)$ and $\Lift(o;H^*) \in T_{os}(k+1)$.

For the former, each summand of $i$ is of the form $c(G-G_{\sigma})$ for some constant $c$, $k$-vertex-labelled graph $G$, and some $\sigma \in S$. Then
$$
\Lift(c(G-G_{\sigma}); H^*) = c[\Lift(G;H^*)-\Lift(G_{\sigma};H^*)] = $$ $$ c[\Lift(G;H^*)-\Lift(G_{\sigma};H_{\sigma'}^*)]
$$
and this is an element of $T_{iso}(k+1)$ since $\Lift(G;H^*)$ is isomorphic to $\Lift(G_{\sigma};H_{\sigma'}^*)$ via the mapping $\sigma'$. An analogous argument proves the claim for $T_{os}(k+1)$.


\end{proof}

To show how we can use this, we provide an illustrative example originally introduced by Orellana and Scott:

\begin{cor}[\cite{ore}, Theorem 4.2]\label{cor:ore}
Suppose that $G$ is a graph with four vertices $u,v,w,z$ such that $uz, uw, zw, vw \in E(G)$ and $uv, zv \notin E(G)$, and suppose further that there exists an automorphism $\phi$ of $G-wz-wu$ such that $\phi(\{u,w\}) = \{v,z\}$ and $\phi(\{v,z\}) = \{u,w\}$. Then $G$ has the same chromatic symmetric function as $G-uw+vz$ (Figure 2 illustrates the relevant induced subgraphs of $G$ and $G-uw+vz$).
\end{cor}

\begin{proof}

We present here a proof different from that in \cite{ore} by showing how to derive the result using Theorem \ref{thm:pass}.

Given the conditions $\phi(\{u,w\}) = \{v,z\}$ and $\phi(\{v,z\}) = \{u,w\}$, there are four possibilities for the tuple $(\phi(u),\phi(v),\phi(w),\phi(z))$. It is easy to verify that two of these are not automorphisms even on these four vertices, so we are left with two possible cases, each of which is an involution on $\{u,v,w,z\}$. 

If $\phi$ exchanges $u$ and $v$, and exchanges $w$ and $z$, then $\phi(uw) = vz$ and $\phi(zw) = zw$, so the graphs $G$ and $G-uw+vz$ are isomorphic, and the result follows trivially.

Thus, from now on we suppose that $\phi$ exchanges $u$ and $z$, and exchanges $v$ and $w$. Then $\phi$ does not extend to an isomorphism of the graphs in question, since $\phi(zw) = uv$, and $zw \in E(G)$ but $uv \notin E(G-uw+vz)$.

Letting $H = G-uw+vz$, we wish to show that $G - H \in Ker(X)$. To take advantage of our assumption, we instead view each of $G$ and $H$ as $3$-vertex-labelled graphs on vertex set $V_1 \sqcup V_2 \sqcup V_3$, where $V_1 = \{u,z\}$, $V_2 = \{v,w\}$, and $V_3 = V(G)\bk\{V_1 \cup V_2\}$. Then by Corollary \ref{cor:proj}, it is sufficient to show that as $3$-vertex-labelled graphs, $G - H \in Ker(X_3)$. By Theorem \ref{thm:pass} and our assumption about the existence of $\phi$, it is sufficient to show that we may express the linear combination of four-vertex graphs depicted in Figure 2 as an element of $Ker(X_2)$ using only the triangular modular relation and elements of $T_{iso}(2)$ that arise from the restricted map $\phi$ swapping $u$ with $z$ and $v$ with $w$ simultaneously. This is done in Figure 3, where we note that the last two graphs are isomorphic via $\phi$, completing the proof.

\end{proof}

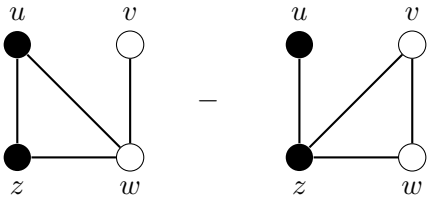
\begin{figure}[hbt]
\begin{center}
  \begin{tikzpicture}[scale=1.5]
    \node[label=below:{$z$}, fill=black, circle] at (0, 0)(1){};
    \node[label=above:{$u$}, fill=black, circle] at (0, 1)(2){};
    \node[label=above:{$v$}, draw, circle] at (1, 1)(3){};
    \node[label=below:{$w$}, draw, circle] at (1, 0)(4){};

    \draw[black, thick] (2) -- (1);
    \draw[black, thick] (1) -- (4);
    \draw[black, thick] (3) -- (4);
    \draw[black, thick] (2) -- (4);

    \draw (1.5, 0.5) coordinate (MI) node[right] { $\bf{-}$ };
    
    \node[label=below:{$z$}, fill=black, circle] at (2.5, 0)(5){};
    \node[label=above:{$u$}, fill=black, circle] at (2.5, 1)(6){};
    \node[label=above:{$v$}, draw, circle] at (3.5, 1)(7){};
    \node[label=below:{$w$}, draw, circle] at (3.5, 0)(8){};
    
    \draw[black, thick] (6) -- (5);
    \draw[black, thick] (7) -- (5);
    \draw[black, thick] (7) -- (8);
    \draw[black, thick] (8) -- (5);

  \end{tikzpicture}
\end{center}
\label{fig:x2}
\caption{The desired element of $Ker(X_2)$ to be extended via Corollary \ref{cor:ore}.}
\end{figure}

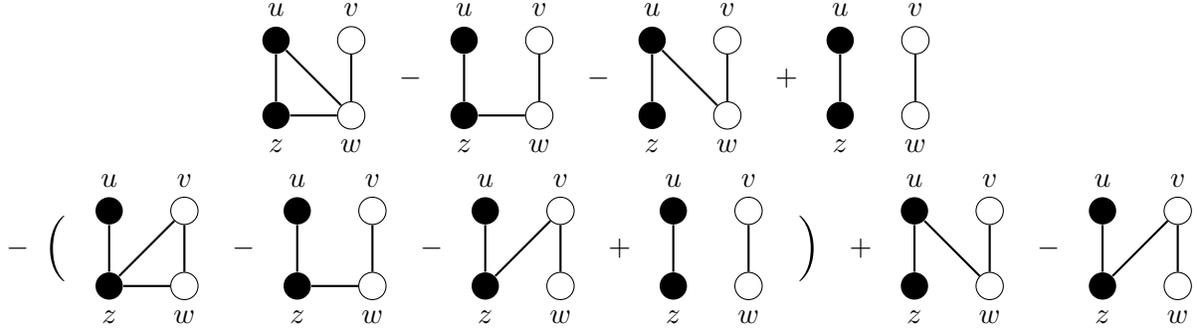
\begin{figure}[hbt]
\begin{center}
  \begin{tikzpicture}[scale=1]
    \node[label=below:{$z$}, fill=black, circle] at (0, 0)(1){};
    \node[label=above:{$u$}, fill=black, circle] at (0, 1)(2){};
    \node[label=above:{$v$}, draw, circle] at (1, 1)(3){};
    \node[label=below:{$w$}, draw, circle] at (1, 0)(4){};

    \draw[black, thick] (2) -- (1);
    \draw[black, thick] (1) -- (4);
    \draw[black, thick] (3) -- (4);
    \draw[black, thick] (2) -- (4);

    \draw (1.5, 0.5) coordinate (MI) node[right] { $\bf{-}$ };
    
    \node[label=below:{$z$}, fill=black, circle] at (2.5, 0)(5){};
    \node[label=above:{$u$}, fill=black, circle] at (2.5, 1)(6){};
    \node[label=above:{$v$}, draw, circle] at (3.5, 1)(7){};
    \node[label=below:{$w$}, draw, circle] at (3.5, 0)(8){};
    
    \draw[black, thick] (6) -- (5);
    \draw[black, thick] (7) -- (8);
    \draw[black, thick] (8) -- (5);
    
    \draw (4, 0.5) coordinate (MI2) node[right] { $\bf{-}$ };
    
    \node[label=below:{$z$}, fill=black, circle] at (5, 0)(z){};
    \node[label=above:{$u$}, fill=black, circle] at (5, 1)(u){};
    \node[label=above:{$v$}, draw, circle] at (6, 1)(v){};
    \node[label=below:{$w$}, draw, circle] at (6, 0)(w){};
    
    \draw[black, thick] (u) -- (z);
    \draw[black, thick] (v) -- (w);
    \draw[black, thick] (u) -- (w);
    
    \draw (6.5, 0.5) coordinate (PL) node[right] { $\bf{+}$ };
    
    \node[label=below:{$z$}, fill=black, circle] at (7.5, 0)(z2){};
    \node[label=above:{$u$}, fill=black, circle] at (7.5, 1)(u2){};
    \node[label=above:{$v$}, draw, circle] at (8.5, 1)(v2){};
    \node[label=below:{$w$}, draw, circle] at (8.5, 0)(w2){};
    
    \draw[black, thick] (u2) -- (z2);
    \draw[black, thick] (v2) -- (w2);

  \end{tikzpicture}
   \begin{tikzpicture}[scale=1]
   
   \draw (-1.5, 0.5) coordinate (MIM) node[right] { $\bf{-}$ };
   
    \draw (-1, 0.5) coordinate (LP) node[right] { $\bf{\text{\Huge (}}$ };
   
    \node[label=below:{$z$}, fill=black, circle] at (0, 0)(1){};
    \node[label=above:{$u$}, fill=black, circle] at (0, 1)(2){};
    \node[label=above:{$v$}, draw, circle] at (1, 1)(3){};
    \node[label=below:{$w$}, draw, circle] at (1, 0)(4){};

    \draw[black, thick] (2) -- (1);
    \draw[black, thick] (1) -- (4);
    \draw[black, thick] (3) -- (4);
    \draw[black, thick] (1) -- (3);

    \draw (1.5, 0.5) coordinate (MI) node[right] { $\bf{-}$ };
    
    \node[label=below:{$z$}, fill=black, circle] at (2.5, 0)(5){};
    \node[label=above:{$u$}, fill=black, circle] at (2.5, 1)(6){};
    \node[label=above:{$v$}, draw, circle] at (3.5, 1)(7){};
    \node[label=below:{$w$}, draw, circle] at (3.5, 0)(8){};
    
    \draw[black, thick] (6) -- (5);
    \draw[black, thick] (7) -- (8);
    \draw[black, thick] (8) -- (5);
    
    \draw (4, 0.5) coordinate (MI2) node[right] { $\bf{-}$ };
    
    \node[label=below:{$z$}, fill=black, circle] at (5, 0)(z){};
    \node[label=above:{$u$}, fill=black, circle] at (5, 1)(u){};
    \node[label=above:{$v$}, draw, circle] at (6, 1)(v){};
    \node[label=below:{$w$}, draw, circle] at (6, 0)(w){};
    
    \draw[black, thick] (u) -- (z);
    \draw[black, thick] (v) -- (w);
    \draw[black, thick] (v) -- (z);
    
    \draw (6.5, 0.5) coordinate (PL) node[right] { $\bf{+}$ };
    
    \node[label=below:{$z$}, fill=black, circle] at (7.5, 0)(z2){};
    \node[label=above:{$u$}, fill=black, circle] at (7.5, 1)(u2){};
    \node[label=above:{$v$}, draw, circle] at (8.5, 1)(v2){};
    \node[label=below:{$w$}, draw, circle] at (8.5, 0)(w2){};
    
    \draw[black, thick] (u2) -- (z2);
    \draw[black, thick] (v2) -- (w2);
    
    \draw (9, 0.5) coordinate (RP) node[right] { $\bf{\text{\Huge )}}$ };

  \end{tikzpicture}
  \begin{tikzpicture}[scale=1]
  
  \draw (-1, 0.5) coordinate (PL) node[right] { $\bf{+}$ };
  
  \node[label=below:{$z$}, fill=black, circle] at (0, 0)(z){};
    \node[label=above:{$u$}, fill=black, circle] at (0, 1)(u){};
    \node[label=above:{$v$}, draw, circle] at (1, 1)(v){};
    \node[label=below:{$w$}, draw, circle] at (1, 0)(w){};
    
    \draw[black, thick] (u) -- (z);
    \draw[black, thick] (v) -- (w);
    \draw[black, thick] (u) -- (w);

    \draw (1.5, 0.5) coordinate (MI) node[right] { $\bf{-}$ };
    
    \node[label=below:{$z$}, fill=black, circle] at (2.5, 0)(5){};
    \node[label=above:{$u$}, fill=black, circle] at (2.5, 1)(6){};
    \node[label=above:{$v$}, draw, circle] at (3.5, 1)(7){};
    \node[label=below:{$w$}, draw, circle] at (3.5, 0)(8){};
    
    \draw[black, thick] (6) -- (5);
    \draw[black, thick] (7) -- (8);
    \draw[black, thick] (7) -- (5);

  \end{tikzpicture}
\end{center}
\label{fig:kerx2}
\caption{The combination in Figure 2 expressed in appropriate kernel-form presentation.}
\end{figure}

Let us briefly summarize what this example illustrates. Suppose that $L = c_1H_1 + \dots + c_mH_m$ is an element of $Ker(X)$ in which every graph has the same vertex set $V$. In previous work by the first and last authors \cite{modular}, we give a necessary and sufficient condition for $L$ to be a \emph{modular relation}, meaning that it satisfies the stronger property that for any graph $G$ with $V(G) \supseteq V(H)$ we have that $Ext(L;G) \in Ker(X)$, providing a characterization of when we can universally extend certain elements of $Ker(X)$ to larger ones. 

But there are some natural instances where we do not need our kernel elements to have a universal extension to any graph, but simply to a sufficiently nice class of graphs. Consider the linear combination $H_1-H_2$ from Figure 2 as an element of $Ker(X)$ without the vertex groupings. It can be shown that $H_1-H_2$ is not a modular relation (in fact for two graphs $G_1$ and $G_2$, $G_1-G_2$ is never a modular relation unless $G_1 = G_2$). However, with Theorem \ref{thm:pass}, we can nonetheless determine a class of graphs with certain symmetries such that we may lift $H_1-H_2$ to larger elements of $Ker(X)$ using graphs in the class.

\subsection{Homogeneous Partitions}

An important special case of Theorem \ref{thm:pass} is that whenever $L \in Ker(X_k)$ is a linear combination of graphs with vertex set $V_1 \sqcup \dots \sqcup V_k$, and $H^*$ is a $(k+1)$-vertex-labelled graph with vertex set $V_1 \sqcup \dots \sqcup V_{k+1}$ that is fixed under \emph{every} permutation in $S_{V_1} \times \dots \times S_{V_k} \times Id_{V_{k+1}}$ (where $Id_{V_{k+1}}$ is the identity permutation of $V_{k+1}$), then $\Lift(L;H^*)$ is \emph{always} an element of $Ker(X_{k+1})$.

\begin{definition}
    In a graph $G$, a partition $V(G) = V_1 \sqcup \dots \sqcup V_k \sqcup W$ is called \emph{homogeneous} if for every $i \in \{1,2,\dots,k\}$ and every $w \in W$, we have that $w$ is either complete or anticomplete to $V_i$ (as defined in Section 2.2). In this case, we say that $V_1, \dots, V_k$ form a \emph{homogeneous collection} in $G$. In the particular case $k=1$ we call $V_1$ a \emph{homogeneous set}, and when $k=2$ we call $V_1 \sqcup V_2$ a \emph{homogeneous pair}.
\end{definition}

Note that whether $V_1 \sqcup \dots \sqcup V_k$ is a homogeneous collection only depends on edges from the $V_i$ to $W$, and not on edges with both endpoints in $W$ or with both endpoints among the $V_i$. 

From the above discussion and Theorem \ref{thm:pass}, we may easily derive the following corollary.

\begin{cor}\label{cor:homog}

Let $L = \sum c_iH_i \in Ker(X_k)$ be a linear combination of graphs with common vertex set $V_1 \sqcup \dots \sqcup V_k$, and let $H^*$ be a $(k+1)$-vertex-labelled graph with vertex set $V(H_i) \sqcup V_{k+1}$ such that $V_1 \sqcup \dots \sqcup V_k$ is a homogeneous collection in $H^*$. Then $\Lift(L;H^*)$ is an element of $Ker(X_{k+1})$.
\end{cor}

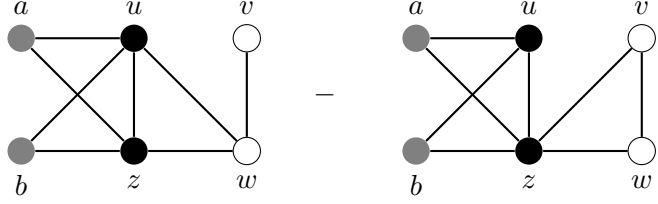
\begin{figure}[hbt]
\begin{center}
  \begin{tikzpicture}[scale=1.5]
    \node[label=below:{$z$}, fill=black, circle] at (0, 0)(1){};
    \node[label=above:{$u$}, fill=black, circle] at (0, 1)(2){};
    \node[label=above:{$v$}, draw, circle] at (1, 1)(3){};
    \node[label=below:{$w$}, draw, circle] at (1, 0)(4){};
    
    \node[label=above:{$a$}, fill=gray, circle] at (-1, 1)(a){};
    \node[label=below:{$b$}, fill=gray, circle] at (-1, 0)(b){};

    \draw[black, thick] (2) -- (1);
    \draw[black, thick] (1) -- (4);
    \draw[black, thick] (3) -- (4);
    \draw[black, thick] (2) -- (4);
    \draw[black, thick] (b) -- (1);
    \draw[black, thick] (b) -- (2);
    \draw[black, thick] (1) -- (a);
    \draw[black, thick] (2) -- (a);

    \draw (1.5, 0.5) coordinate (MI) node[right] { $\bf{-}$ };
    
    \node[label=below:{$z$}, fill=black, circle] at (3.5, 0)(5){};
    \node[label=above:{$u$}, fill=black, circle] at (3.5, 1)(6){};
    \node[label=above:{$v$}, draw, circle] at (4.5, 1)(7){};
    \node[label=below:{$w$}, draw, circle] at (4.5, 0)(8){};
    
    \node[label=above:{$a$}, fill=gray, circle] at (2.5, 1)(c){};
    \node[label=below:{$b$}, fill=gray, circle] at (2.5, 0)(d){};
    
    \draw[black, thick] (6) -- (5);
    \draw[black, thick] (7) -- (5);
    \draw[black, thick] (7) -- (8);
    \draw[black, thick] (8) -- (5);
    \draw[black, thick] (c) -- (5);
    \draw[black, thick] (c) -- (6);
    \draw[black, thick] (d) -- (6);
    \draw[black, thick] (d) -- (5);

  \end{tikzpicture}
\end{center}
\label{fig:homog}
\caption{An example of nonisomorphic graphs with equal chromatic $3$-multisymmetric function arising from lifting an element of $Ker(X_2)$ to $Ker(X_3)$ with a homogeneous pair. Note that $a$ and $b$ are complete to $\{u,z\}$ and anticomplete to $\{v,w\}$. Deleting $b$ from both graphs and removing the labels yields the smallest example of nonisomorphic graphs with equal chromatic symmetric function \cite{stanley}.}
\end{figure}

Thus, if we are working within a graph class in which we can find homogeneous collections, we can simplify or reduce problems for chromatic symmetric functions in that class.

Implicitly, this was part of the approach taken by Guay-Paquet in \cite{guay} when reducing the Stanley-Stembridge conjecture that incomparability graphs of $(3+1)$-free posets are $e$-positive: 

\begin{theorem}[\cite{guay}, Theorem 5.1]\label{thm:guay}

Suppose that $G$ is the incomparability graph of a $(3+1)$-free poset. Then $X_G$ may be written as a convex combination of the chromatic symmetric functions of graphs that are incomparability graphs of $(3+1)$- and $(2+2)$-free posets.

\end{theorem}

This theorem demonstrates that to prove the Stanley-Stembridge conjecture, it is sufficient to determine the $e$-positivity of incomparability graphs of $(3+1)$- and $(2+2)$-free posets, or equivalently unit interval graphs. 

Guay-Paquet's proof of Theorem \ref{thm:guay} relied heavily on the structure theorem for incomparability graphs of $(3+1)$-free posets established in \cite{clawposet}. In what follows we summarize his proof approach in \cite{guay}, rewriting relevant portions using the methodology developed thus far, in the process creating a proof that does not require reference to the stronger structure theorem of \cite{clawposet}. This will be summarized in five steps.

We first give a brief overview of poset notation for clarity: a poset $P = (V,<_P)$ consists of a set $V$ of vertices, and a partial order $<_P$ on $V$. Elements $v,w \in P$ are \emph{comparable} if $v <_P w$ or $w <_P v$, and otherwise they are \emph{incomparable}. The incomparability graph of a poset $P$ is a simple graph with the same vertex set $V$, and where two distinct vertices $v$ and $w$ are adjacent if and only if they are incomparable in $P$. 

The \emph{$(3+1)$ poset} has vertices $a,b,c,d$ with partial order $a <_P b <_P c$ (and $d$ incomparable with $a,b,c$). The \emph{$(2+2)$ poset} has the same vertex set, but with partial order $a <_P c$ and $b <_P d$, and no other relations. The notions of poset isomorphism, induced posets, and $P$-free posets are exactly analogous to the corresponding graph notions.

\begin{enumerate}
    \item Suppose that $G$ is the incomparability graph of a $(3+1)$-free poset $P$, and that $G$ contains an induced $C_4$ (meaning that $P$ contains an induced $(2+2)$ poset). Then the vertices of this $C_4$ may be labelled $a,b,c,d$ such that there exists a homogeneous pair of cliques $V_1 \sqcup V_2$ in $G$ with $a,b \in V_1$ and $c,d \in V_2$. This was implicitly demonstrated by Guay-Paquet, Morales, and Rowland in \cite{clawposet} as a consequence of their structure theorem for such graphs; we give a direct proof below from the point of view of the poset $P$.
    
    \begin{lemma}\label{lem:sqcon}
    
    Let $P$ be a $(3+1)$-free poset, and let $a,b,c,d$ be an induced $(2+2)$ poset, with $a <_P c$ and $b <_P d$ (so $a$ is incomparable with $b$ and $d$, and $b$ is incomparable with $c$).
    
    Let $V_1,V_2 \subseteq P$ be maximal with respect to inclusion such that
    \begin{enumerate}
        \item $a,b \in V_1$ and $c,d \in V_2$.
        \item All elements of $V_1$ are pairwise incomparable, and all elements of $V_2$ are pairwise incomparable.
        \item For every nonempty $A \subsetneq V_1$, there exists an induced $(2+2)$ poset with one vertex in $A$, one vertex in $V_1 \bk A$, and two in $V_2$.
        \item For every nonempty $B \subsetneq V_2$, there exists an induced $(2+2)$ poset with one vertex in $B$, one vertex in $V_2 \bk B$, and two vertices in $V_1$.
         \item For every $x \in V_1$ and every $y \in V_2$, either $x <_P y$ or $x$ and $y$ are incomparable.
    \end{enumerate}
    
    Then $V_1$ and $V_2$ are a homogeneous pair of cliques in the incomparability graph $G$ of $P$, meaning that in the poset they satisfy:
    
    \begin{enumerate}[resume]
        \item For every $v \in P \bk (V_1 \cup V_2)$, $v$ is either comparable to every element of $V_1$ or incomparable with every element of $V_1$, and likewise for $V_2$.
    \end{enumerate}
    
    \end{lemma}
    
    Conditions (a), (b), and (f) are the ones we want $V_1$ and $V_2$ to satisfy, and condition (e) will be convenient for later steps. 
    
    Conditions (c) and (d) are the poset version of the notion of a \emph{square-connected} homogeneous pair of cliques in a graph, meaning a homogeneous pair of cliques $V_1 \sqcup V_2$ such that for every partition of one of the $V_i$ into two nonempty parts $A \sqcup B$ for $i \in \{1,2\}$, there exists an induced $C_4$ with one vertex in $A$, one vertex in $B$, and two vertices in $V_{3-i}$. 
    
    The notion of square-connected homogeneous cliques occurs naturally in the study of claw-free perfect graphs \cite{cemil}, where they are used to find homogeneous pairs. Here, it is used to provide necessary structure for the proof.
    
    \begin{proof}
    
    Throughout, we will switch between the graph and poset perspectives. In particular, when we say that $v, w \in V$ are neighbors, we mean they are adjacent in $G$, so are incomparable in $P$.
    
    Note that there is at least one choice of $(V_1,V_2)$ that satisfies (a), (b), (c), (d), and (e) above, since we may take $V_1 = \{a,b\}$ and $V_2 = \{c,d\}$. Among all such choices, choose one such that $V_1$ and $V_2$ are maximal with respect to inclusion. We prove that such a choice satisfies (f).
    
    We suppose otherwise for a contradiction; without loss of generality, we assume there exists $v \in V \bk (V_1 \cup V_2)$ such that $v$ has at least one neighbor and at least one nonneighbor in $V_2$ (the case replacing $V_2$ by $V_1$ can be done analogously by reversing $<_P$).
    
    Note that there do not exist $y,y' \in V_2$ such that $y <_P v <_P y'$ since $y$ and $y'$ are incomparable by assumption. 
    
    Suppose first that all of the nonneighbors of $v$ in $V_2$ are $<_P$-smaller than $v$. Then applying property (d), choosing $B$ to be the set of nonneighbors of $v$, there exist $y \in B$, $y' \in V_2 \bk B$, and $x,x' \in V_1$ that form an induced $(2+2)$ with $x <_P y$ and $x' <_P y'$. But then $x, y, v$ form a chain, and together with $y'$ the four vertices form an induced $(3+1)$, which is a contradiction. 
    
    Thus, $v$ is $<_P$-smaller than all of its nonneighbors in $V_2$. Applying (d) as above again, we find $x,x',y,y'$ such that $v <_P y$ and $x <_P y$, and $x, v, y$ all incomparable with $y'$. Then $x$ is not comparable with $v$, as otherwise $x,v,y,y'$ would form an induced $(3+1)$. Thus, $v$ has at least one neighbor in $V_1$.
    
    Suppose now that $v$ has at least one nonneighbor in $V_1$. By the same proof as above, $v$ is not $<_P$-smaller than its nonneighbors in $V_1$, so $v$ must be larger than all of its nonneighbors in $V_1$. Let $N(V_1)$ and $N(V_2)$ be the sets of nonneighbors of $v$ in $V_1$ and $V_2$ respectively. Then for $x \in N(V_1)$ and $y \in N(V_2)$, we have $x <_P y$ by transitivity. On the other hand, for $x \in N(V_1)$ and $y' \in V_2 \bk N(V_2)$, if $y'$ is incomparable with $v$, and it is not the case that $y'$ is incomparable to $x$, then for every $y \in N(V_2)$, we have that $x,v,y,y'$ forms an induced $(3+1)$. Thus, for all such $x$ and $y'$, we have $x <_P y'$ by property (e). But this means that property (c) is violated when choosing $A = N(V_1)$, since there are no edges between $N(V_1)$ and $V_2$, a contradiction.
    Thus, $v$ has no nonneighbors in $V_1$.
    
    Then it is easy to verify that the pair $(V_1 \cup \{v\}, V_2)$ satisfies properties (a), (b), (d), and (e).
    
    We now show that it also satisfies (c). Clearly (c) is satisfied whenever both $A$ and $(V_1 \cup \{v\}) \bk A$ contain vertices other than $v$ since $(V_1,V_2)$ satisfies (c), so it suffices to show that an induced $(2+2)$-poset exists containing $v$, a vertex of $V_1$, and two vertices of $V_2$.
    Applying property (d) to $(V_1,V_2)$ as before with $B = N(V_2)$, we obtain $x, x', y, y'$ with $x' <_P y'$ and $x'$ incomparable with $y$. Then as $v <_P y$ and $v$ is incomparable with $y'$, the vertices $v,x',y,y'$ form an induced $(2+2)$.
    
    Therefore, $(V_1 \cup \{v\}, V_2)$ satisfies (a), (b), (c), (d), and (e), contradicting the maximality of $(V_1,V_2)$. It follows that our choice of a vertex $v$ which has both a neighbor in $V_2$ and a nonneighbor in $V_2$ is impossible. Applying the proof again for $V_1$ with $<_P$ reversed, the conclusion follows that $(V_1,V_2)$ satisfies (f).
    
    \end{proof}
    
    
    
     
     
     \item We have found a pair of homogeneous cliques $V_1$ and $V_2$ in $G$. Let $|V_1| = m$ and $|V_2| = n$, and suppose without loss of generality that $m \leq n$ (the other case is analogous). Let the vertices of $V_1 \sqcup V_2$ be $v_1, \dots, v_m, w_1, \dots, w_n$. Define the graphs $G_k$ for $k \in \{0,1,\dots,m\}$ to have vertex set $V_1 \sqcup V_2$ and edge set
    $\{v_iv_j: 1 \leq i < j \leq m\} \cup \{w_iw_j: 1 \leq i < j \leq n\} \cup \{v_iw_j: 0 \leq i \leq k, \textnormal{ all } j\}$. 
    
    Guay-Paquet uses the triangular modular relation of $Ker(X)$ and a novel linear algebraic argument to show that the chromatic symmetric function of $G|_{V_1 \sqcup V_2}$ may be written as a convex combination of the chromatic symmetric functions of the $G_k$, such that the coefficient of $X_{G_k}$ in $X_{G|_{V_1 \sqcup V_2}}$ is equal to the probability that a randomly chosen map $L: \{1,\dots,m\} \rightarrow \{1,\dots,n\}$ will satisfy that exactly $k$ of the pairs $v_jw_{L(j)}$ are edges of $H$ \cite[Sections 4-5]{guay}. 
    
    It is not hard to verify that the same argument holds entirely analogously in the $2$-vertex-labelled setting for the coefficient of $X_{(G_k,2)}$ in $X_{(G|_{V_1 \sqcup V_2, 2)}}$ where the vertices are labelled in the natural way; the only part that requires additional verification is the proof of \cite[Proposition 4.1 (i)]{guay}, where in the original proof the intermediate step is taken of reducing using elements of $Ker(X)$ to express $X_{G|_{V_1 \sqcup V_2}}$ as a linear combination of graphs with vertex set $V_1 \sqcup V_2$ whose edges between $V_1$ and $V_2$ form a matching; but it is easy to check that this process only uses elements of $Ker(X)$ that are also present as elements of $Ker(X_2)$ when viewing the graphs as $2$-vertex-labelled.

    \item We lift the above to a convex combination of $3$-vertex-labelled graphs and apply Corollary \ref{cor:homog} and Corollary \ref{cor:proj} to find that the analogous relation holds for the overall graph $G$.
    
    \item We define $H_k$ as the graph formed by taking $G$ and replacing $G|_{V_1 \sqcup V_2}$ by $G_k$. In the above steps, we have shown that the chromatic symmetric function of $G$ may be written as a linear combination of the chromatic symmetric functions of the $H_k$. We now show that the $H_k$ are incomparability graphs of $(3+1)$-free posets.
    
    \begin{lemma}\label{lem:3+1}
    
    Let $G$ be the incomparability graph of a $(3+1)$-free poset $P$, and define $H_k$ as above. Then $H_k$ is also the incomparability graph of a $(3+1)$-free poset.
    
    \end{lemma}
    
    \begin{proof}
    
     First we verify that $H_k$ remains an incomparability graph. Let $Q = (V,<_Q)$ have the same vertex set as $P$ formed by letting $<_Q$ be an asymmetric relation on vertex pairs that is the same as $<_P$, except that in $V_1 \sqcup V_2$, we remove all relations between vertex pairs that are now edges in the graph $H_k$, and for all nonedges $xy$ with $x \in V_1$ and $y \in V_2$, we let $x <_Q y$. 
     
     First, we verify that $Q$ is a poset. Clearly $<_Q$ is asymmetric and reflexive if $<_P$ is, so we only need to verify that transitivity holds. Suppose otherwise, that we have vertices $x,y,z$ such that $x <_Q y <_Q z$ but $x \nless_Q z$. This cannot happen if all three vertices lie in $V_1 \sqcup V_2$ since there is no chain $x <_Q y <_Q z$ of three elements. Likewise, if zero or one of the vertices lie in $V_1 \sqcup V_2$, or if two of the vertices lie in the same $V_i$, since then no relations among these vertices have changed from $P$ to $Q$, contradicting that $P$ is a poset. 
     
     Thus, we may assume that among $\{x,y,z\}$, there is one vertex in $V_1$, one in $V_2$, and one outside of $V_1 \sqcup V_2$; call these $u,v,w$ respectively (so $\{u,v,w\} = \{x,y,z\}$). As before, note that if $u <_Q w$, then also $u' <_Q w$ for every $u' \in V_1$ by homogeneity (property (f) of $(V_1,V_2)$) and the fact that vertices of $V_1$ are pairwise incomparable (property (b)), and analogously for other relationships between $w$ and either $u$ or $v$. Note that all relations involving $w$ are unchanged between $P$ and $Q$.
     
    Furthermore, in $P$, we have $a,b \in V_1$ and $c \in V_2$ such that $a <_P c$ and $b \nless_P c$, so we may find vertices $u' \in V_1$ and $v' \in V_2$ such that all pairwise relations between $u',v',w$ hold in $P$ if and only if the corresponding relations hold in $Q$. But then if $\{u,v,w\}$ violate transitivity in $Q$, $\{u',v',w\}$ violates transitivity in $P$, contradicting that $P$ is a poset.

     Thus, this newly formed $Q$ is a poset, and $H_k$ is its incomparability graph.
    
    

    It remains to show that $Q$ is $(3+1)$-free. Suppose otherwise, that there is $X \subseteq V$ with $|X| = 4$ such that $Q|_X$ is an induced $(3+1)$. Clearly $|X \cap (V_1 \sqcup V_2)|$ is not equal to $0,1,$ or $4$, since the original poset $P$ was $(3+1)$-free, and $Q|_{V_1 \sqcup V_2}$ is now $(3+1)$-free.
    
    Suppose first that $|X \cap (V_1 \sqcup V_2)| = 3$. Then some two vertices $v_1, v_2 \in X \cap (V_1 \sqcup V_2)$ lie in the same $V_i$. Since every vertex outside $V_1 \sqcup V_2$ is either less than both, greater than both, or adjacent to both of $v_1,v_2$, and in each case we may verify that $Q|_X$ is not an induced $(3+1)$.
    
    It remains to check the case when $|X \cap (V_1 \sqcup V_2)| = 2$. Using the above reasoning, among the two vertices of this intersection, there is one in each $V_i$. Thus, let $x \in X \cap V_1, y \in X \cap V_2$, and $v,w \in X \bk \{V_1 \sqcup V_2\}$ be given. We suppose for the sake of contradiction there is an induced $(3+1)$ formed by these vertices. 
    
    We proceed similarly to the above argument for proving the transitivity of $<_Q$. Since $a, b \in V_1$ and $c \in V_2$ such that $a <_P b$ and $a \nless_P c$ using the vertices of the $C_4$, we may choose $y' \in V_2$ such that $a <_P y'$ if and only if $x <_Q y$. All other relations in $Q$ amongst the vertices of $X$ are identical to those in $P$ since all such pairs have at least one vertex outside of $V_1 \sqcup V_2$. But then if $X$ is an induced $(3+1)$ in $Q$, then $(X \bk \{x,y\}) \cup \{a,y'\}$ is an induced $(3+1)$ in $P$, contradicting that $P$ is $(3+1)$-free.
    
    Thus, it follows that $Q$ remains $(3+1)$-free.
    
    \end{proof}
    
    The above proof demonstrates how homogeneous pairs are useful for preserving forbidden induced subgraphs, and similar arguments were used in \cite{cemil, perfect}.
    
    \item We prove that no new induced $C_4$ is introduced in $H_k$, so that repeatedly applying this process to the resulting graphs eventually terminates.
    
    \begin{lemma}\label{lem:2+2}
    Let $G,P,H_k,Q$ be as in Lemma \ref{lem:3+1}. Then $Q$ has strictly fewer induced $(2+2)$ posets than $P$.
    \end{lemma}
    
    \begin{proof}
    
    Since $\{a,b,c,d\}$ now do not form an induced $(2+2)$ in $Q$, it suffices to show that there is no $X \subseteq V$ with $|V| = 4$ such that $Q|_X$ is an induced $(2+2)$ but $P|_X$ is not. As above, clearly $|X \cap (V_1 \sqcup V_2)| \in \{2,3\}$.
    
    If $|X \cap (V_1 \sqcup V_2)| = 3$, then there is some $V_i$ such that $|X \cap V_i| \geq 2$. But then the vertex of $X \bk (V_1 \sqcup V_2)$ is smaller than, larger than, or incomparable with all vertices in $X \cap V_i$, and it is easy to check that then $Q|_X$ is not an induced $(2+2)$.
    
    It remains to check the case when $|X \cap (V_1 \sqcup V_2)| = 2$. By the argument above, we may assume that $X$ contains $x \in V_1, y \in V_2$, and $v,w \in V \bk \{V_1 \sqcup V_2\}$. We check two cases:
    
    \begin{itemize}
        \item Case 1: $x <_Q y$. Without loss of generality suppose that $v <_Q w$, and $v$ and $w$ are incomparable with $x$ and $y$. By homogeneity it follows that $v$ and $w$ are then incomparable with all of $V_1 \sqcup V_2$. Then consider $(V_1 \cup \{v\}, V_2 \cup \{w\})$ in $P$. Clearly this pair satisfies properties (a), (b), and (e) of Lemma \ref{lem:sqcon}, and it is easy to check that (c) and (d) are satisfied as well as in the proof of Lemma \ref{lem:sqcon}. But this violates the maximality of $(V_1,V_2)$, a contradiction.
        \item Case 2: $x$ and $y$ are incomparable, and thus $v$ and $w$ are incomparable. Suppose first that $w <_Q x$. Then $w$ is incomparable with $y$ and also $w <_P x$. By property (c) of Lemma \ref{lem:sqcon} we can find $y' \in V_2$ such that $x <_P y'$, while by homogeneity $w$ is incomparable with $y'$, but this contradicts $w <_P x <_P y'$.
        
        Analogous arguments show that none of $v <_Q x$, $v >_Q y$, or $w >_Q y$ hold.
        
        So, we may assume that $v <_Q y$ and $w >_Q x$. Then as in the previous part, we may check that the pair $(V_1 \cup \{v\}, V_2 \cup \{x\})$ satisfies all properties other than (f) of Lemma \ref{lem:sqcon}, contradicting the maximality of $(V_1, V_2)$.
    \end{itemize}
    \end{proof}
    
    Thus, each $H_k$ has strictly fewer induced $C_4$s than $G$, and so by repeatedly applying these steps we eventually write $G$ as a convex combination of incomparability graphs of $(3+1)$- and $(2+2)$-free posets.
    
    \end{enumerate}

\section{Further Directions}

The approach outlined at the end of the previous section could be applied to a number of different problems in the theory of chromatic symmetric functions.

First, it seems plausible that the proof above in Lemma \ref{lem:sqcon} (that each $C_4$ in the incomparability graph of a $(3+1)$-free poset can be extended to homogeneous pair of cliques) can work in a larger class of graphs, thus giving more examples of $e$-positive graphs when combined with Guay-Paquet's argument in \cite{guay}. For instance, it is not the case that each $C_4$ in a claw-free graph necessarily extends to a homogeneous pair of cliques, but it may be true upon adding a much smaller number of forbidden induced subgraphs than are necessary for the large class of all incomparability graphs.

Second, expanding a graph's chromatic symmetric function into the $e$-basis is equivalent to writing it as a linear combination of graphs which are disjoint unions of cliques. Such graphs are also precisely the set of all graphs that have no induced three-vertex path as noted previously. Perhaps it is possible to find some explicit way of determining $e$-basis coefficients by determining a modifiable structure in unit interval graphs with such an induced subgraph.

Third, the original purpose of studying $Ker(X)$ more closely is to make progress on the tree isomorphism conjecture \cite{stanley}, which purports that if $T$ and $T'$ are nonisomorphic trees then $X_T \neq X_{T'}$. In fact, to the best of the authors' knowledge it is not known whether there are bipartite graphs with equal chromatic symmetric function; bipartite graphs may be particularly natural to view through the lens of chromatic multisymmetric functions.

\bibliographystyle{plain}
\bibliography{bib}

\end{document}